\RequirePackage[l2tabu,orthodox,abort]{nag}

\documentclass[
12pt,
reqno,
tbtags
]
{amsart}

\usepackage[a4paper,left=3cm,right=3cm,top=4cm,bottom=4cm]{geometry}

\usepackage[T1]{fontenc}
\usepackage[utf8x]{inputenc}

\usepackage[strict=true]{csquotes}

\usepackage{graphicx}
\usepackage{amsmath}
\usepackage{amsfonts}
\usepackage{amssymb}
\usepackage{amsthm}
\usepackage{bbm} 
\usepackage{subdepth}

\usepackage{tkz-graph}
\usetikzlibrary{arrows}

\numberwithin{equation}{section}

\theoremstyle{plain}

\newtheorem*{maintheorem}{Main Theorem}

\newtheorem{theorem}{Theorem}[section]
\newtheorem{proposition}[theorem]{Proposition}
\newtheorem{lemma}[theorem]{Lemma}
\newtheorem{corollary}[theorem]{Corollary}

\theoremstyle{definition}

\newtheorem{remark}[theorem]{Remark}

\newcommand{\ind}{\mathbbm{1}}

\newcommand{\N}{\mathbb{N}}
\newcommand{\R}{\mathbb{R}}
\newcommand{\Z}{\mathbb{Z}}

\newcommand{\C}{\mathbb{C}}

\newcommand{\X}{\mathbb{X}}
\newcommand{\Y}{\mathbb{Y}}

\providecommand{\cl}[1]{\mathcal{#1}}

\newcommand{\ca}{\mathcal{A}}
\newcommand{\cB}{\mathcal{B}}
\newcommand{\cg}{\mathcal{G}}
\newcommand{\cv}{\mathcal{V}}
\newcommand{\cn}{\mathcal{N}}
\newcommand{\ce}{\mathcal{E}}

\newcommand{\cP}{\mathcal{P}}

\newcommand{\tiA}{\tilde A}
\newcommand{\tiB}{\tilde B}
\newcommand{\tiG}{\tilde G}
\newcommand{\tiM}{\tilde M}

\newcommand{\cAex}{\mathcal{A}_{\textnormal{ex}}}
\newcommand{\Aex}{A_{\textnormal{ex}}}
\newcommand{\cBex}{\mathcal{B}_{\textnormal{ex}}}
\newcommand{\Bex}{B_{\textnormal{ex}}}

\newcommand{\lin}{\mathcal{L}}

\renewcommand{\phi}{\varphi}
\renewcommand{\epsilon}{\varepsilon}

\newcommand{\e}{\mathrm{e}}

\newcommand{\dd}{\,\mathrm{d}}

\newcommand{\conj}{\overline}

\newcommand{\sump}{\sideset{}{'}\sum}

\DeclareMathOperator{\re}{Re}
\DeclareMathOperator{\im}{Im}
\DeclareMathOperator{\range}{range}

\usepackage{mathtools}
\DeclarePairedDelimiter{\lrp}{(}{)}
\DeclarePairedDelimiter{\lrb}{[}{]}
\DeclarePairedDelimiter{\lrc}{\{}{\}}
\DeclarePairedDelimiter{\set}{\{}{\}}
\DeclarePairedDelimiter{\abs}{\lvert}{\rvert}
\DeclarePairedDelimiter{\norm}{\lVert}{\rVert}

\DeclarePairedDelimiter{\scl}{\langle}{\rangle}
\DeclarePairedDelimiter{\seq}{(}{)}

\newcommand{\dnorm}{\norm{ \, \cdot \, }}

\newcommand{\sem}[1]{\set{#1}}

\makeatletter
\newcommand{\normt}{\@ifstar\@normts\@normt}
\newcommand{\@normts}[1]{%
  \left|\mkern-1.5mu\left|\mkern-1.5mu\left|
   #1
  \right|\mkern-1.5mu\right|\mkern-1.5mu\right|
}
\newcommand{\@normt}[2][]{%
  \mathopen{#1|\mkern-1.5mu#1|\mkern-1.5mu#1|}
  #2
  \mathclose{#1|\mkern-1.5mu#1|\mkern-1.5mu#1|}
}
\makeatother

\renewcommand{\leq}{\leqslant}
\renewcommand{\geq}{\geqslant}

\newcommand{\fa}{\mathfrak{a}}
\newcommand{\fb}{\mathfrak{b}}
\newcommand{\fc}{\mathfrak{c}}

\usepackage{hyperref}

\title {Asymptotic behaviour of fast diffusions on graphs}
\author[A. Gregosiewicz]{Adam Gregosiewicz}

\address{%
  Institute of Mathematics, Polish Academy of Sciences, ul.~Sniadeckich 8,
  00-656 Warsaw, Poland}

\address{%
  Lublin University of Technology,\ ul.~Nadbystrzycka 38A,\ 20-618 Lublin,
  Poland}

\email{a.gregosiewicz@pollub.pl}

\begin{document}

\begin{abstract}
We investigate fast diffusions on finite directed graphs with semipermeable
membranes on vertices.
We prove, in \( L^1 \) and \( L^2 \)-type spaces, that there is a~semigroup of
operators related to the process, and we describe asymptotic behaviour of the
diffusion semigroup as the diffusions' speed increases at the same rate as the
probability of a~particle passing through a~vertex decreases.
In \( L^1 \) case it turns out that the limit process is a~Markov chain on the
vertices of the line graph of the initial graph.
The results are inspired, and in a~way dual to those obtained by A.~Bobrowski in
A.~Ann.~Henri Poincaré (2012) 13: 1501--1510.
\end{abstract}

\maketitle

\section{Introduction}
\label{sec:introduction}

Assume that \( \cg \) is a~directed graph in \( \R^3 \) without loops, and there
is a~Markov process on \( \cg \), which on each edge behaves like Brownian
motion with given variance.
Moreover, assume that each vertex is a~semipermeable membrane with given
permeability coefficients, that is for each vertex there are nonnegative numbers
\( p_{ij} \), describing the probability of a~particle passing through membrane
from the \( i \)-th to the \( j \)-th edge.

In~\cite{Bobrowski-diff} and~\cite{Bobrowski-Morawska} the authors prove that if
the diffusion's speed increases to infinity with the same rate as permeability
coefficients decreases to zero, then we obtain a~limit process which is a~Markov
chain on the vertices of the line graph of \( \cg \), see
Figure~\ref{fig:graph}.
\begin{figure}
\centering
\includegraphics{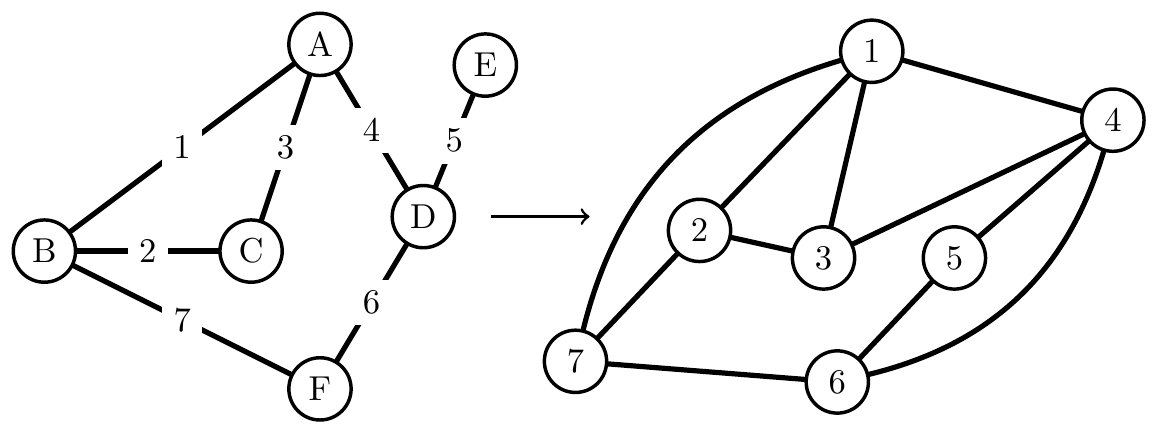}
\caption{Diffussion on graph becomes Markov chain on the vertices of the line
  graph.}
\label{fig:graph}
\end{figure}

The aim of this paper is to prove similar asymptotic result but in a~different
spaces.
In~\cite{Bobrowski-diff,Bobrowski-Morawska} the authors consider the process in
the space of continuous functions on a~graph \( \cg \).
Here we consider \( L^1 \) and \( L^2 \)-type spaces of Lebesgue integrable and
square integrable functions.

The described model is a~special case of an evolution operator acting on
a~graph.
For more such models see~\cite{Mugnolo}.

\subsection{Continuous case}
\label{sec:continuous-case}

As in~\cite{Bobrowski-diff}, let \( \cg = (\cv, \ce) \) be a finite geometric
graph~(see e.g.~\cite[p.~65]{Mugnolo}) without loops, where
\( \cv \subset \R^3 \) is the set of verices, and \( \ce \) is the set of edges
of finite length.
The edges are seen as \( C^1 \) curves connecting vertices.
Let \( N \) be the number of edges and denote
\[
\cn \coloneqq \lrc*{ 1, \dots, N }, \qquad \N \coloneqq \set{ 1,2,\ldots }.
\]
For each \( i \in \cn \), by convention, we call the initial and terminal
vertices of the \( i \)-th edge \( E_i \) its ``left'' and ``right'' endpoints.
We denote them by \( L_i \) and \( R_i \), respectively.
Moreover, for \( i \in \cn \) let \( V_i \) denote vertex \( V \in \cv \) as an
endpoint of the \( i \)-th edge.
If \( V \) is not an endpoint of this edge, we leave \( V_i \) undefined.

Let \( S = \bigcup_{i \in \cn} E_i \) be the disjoint union of the edges.
Notice that there can be many ``copies'' of the same vertex in \( S \), treated
as an endpoint of different edges, since by convention \( V_i \neq V_j \) in
\( S \) for \( i,j \in \cn \), \( i \neq j \).
Then \( S \) is a disconnected compact topological space, and we denote by
\( C(S) \) the space of continuous functions on \( S \) with standard supremum
norm.
We may identify \( f \in C(S) \) with \( (f_i)_{i \in \cn} \), where \( f_i \)
is a member of \( C(E_i) \), the space of continuous functions on the edge
\( E_i \).
The latter space is isometrically isomorphic to the space \( C[0,d_i] \) of
continuous functions defined on the interval \( [0,d_i] \), where \( d_i \) is
the length of the \( i \)-th edge.

Let \( \sigma \in C(S) \) be defined by \( \sigma(p) = \sigma_i \) for
\( i \in \cn \) and \( p \in E_i \), where \( \sigma_i \)'s are given positive
numbers.
Define the operator \( A \) in \( C(S) \) by
\begin{equation}
\label{eq:semigroup-in-C}
Af = \sigma f'',
\end{equation}
on the domain composed of twice continuously differentiable functions,
satisfying the transmission conditions described below.

For each \( i \in \cn \), let \( l_i \) and \( r_i \) be nonnegative real
numbers describing the possibility of passing through the membrane from the
\( i \)-th edge to the edges incident in the left and right endpoints,
respectively.
Also, for \( i,j \in \cn \) such that \( i \neq j \) let \( l_{ij} \) and
\( r_{ij} \) be nonnegative real numbers satisfying
\( \sum_{j \neq i} l_{ij} \leq l_i \) and \( \sum_{j \neq i} r_{ij} \leq r_i \).
The summation here is taken over all \( j \in \cn \) such that \( j \neq i \).
These numbers determine the probability that after filtering through the
membrane from the \( i \)-th edge, a particle will enter the \( j \)-th edge.

By default, if \( E_j \) is not incident with \( L_i \), we put
\( l_{ij} = 0 \).
In particular, by convention \( l_{ij} f(V_j) = 0 \) for \( f \in C(S) \), if
\( V_j \) is not defined.
The same remark concerns \( r_{ij} \).
With these notations, the transmission conditions mentioned above are as
follows: if \( L_i = V\), then
\begin{equation}
\label{eq:transmission-condition-1}
f'(V_i) = l_i f(V_i) - \sum_{j \neq i} l_{ij} f(V_j),
\end{equation}
where \( f'(V_i) \) is the right-hand derivative of \( f \) at \( V_i \), and if
\( R_i = V \), then
\begin{equation}
\label{eq:transmission-condition-2}
-f'(V_i) = r_i f(V_i) - \sum_{j \neq i} r_{ij} f(V_j),
\end{equation}
where \( f'(V_i) \) is the left-hand derivative of \( f \) at \( V_i \).

It is showed in~\cite{Bobrowski-diff} that the operator \( A \) generates a
\emph{Feller semigroup} \( \sem{ \e^{tA} }_{t \geq 0} \) in \( C(S) \).
This means that \( \sem{ \e^{tA} }_{t \geq 0} \) is a~strongly continuous
semigroup of nonnegative contractions, that is for all \( t \geq 0 \) we have
\( \norm{ \e^{tA} }_{\lin(C(S))} \leq 1 \), and \( \e^{tA} f \geq 0 \), provided
that \( f \in C(S) \) is nonnegative.
Here, \( \norm{ \, \cdot \, }_{\lin(C(S))} \) is the operator norm related to
the standard supremum norm in \( C(S) \).
Moreover, the semigroup is \emph{conservative}, that is
\( \e^{tA} \ind_S = \ind_S \), where \( \ind_S \equiv 1 \) on \( S \), if and
only if \( \sum_{j \neq i} l_{ij} = l_i \) and
\( \sum_{j \neq i} r_{ij} = r_i \) for \( i \in \cn \).

Let \( \seq{\kappa_n}_{n \in \N} \) be a~nondecreasing sequence of positive
numbers converging to infinity, and let operators \( A_n \) be defined
by~\eqref{eq:semigroup-in-C} with \( \sigma \) replaced by
\( \kappa_n \sigma \), that is
\[
A_n f = \kappa_n \sigma f'',
\]
with domain \( D(A_n) \) composed of twice continuously differentiable functions
on \( S \) satisfying transmission
conditions~\eqref{eq:transmission-condition-1}
and~\eqref{eq:transmission-condition-2} with permeability coefficients
\textup{(}that is all \( l_i \), \( r_i \), \( l_{ij} \) and
\( r_{ij} \)'s\textup{)} divided by \( \kappa_n \).
The following is proved in~\cite[Theorem~2.2]{Bobrowski-diff}.

\begin{theorem}
\label{thm:convergence-in-C}
For every \( t > 0 \) and \( f \in C(S) \) it follows that
\[
\lim_{n \to +\infty} \e^{tA_n} f = \e^{tQ} Pf
\]
in \( C(S) \), where \( P \) is the projection of \( C(S) \) onto the space
\( C_0(S) \) of functions that are constant on each edge, given by
\( Pf = \lrp[\big]{ d_i^{-1} \int_{E_i} f }_{i \in \cn} \), while \( Q \) is the
operator in \( C_0(S) \) which may be identified with the matrix
\( (q_{ij})_{i,j \in \cn} \) with
\( q_{ij} = \sigma_i d_i^{-1} (l_{ij} + r_{ij}) \) for \( i \neq j \) and
\( q_{ii} = -\sigma_i d_i^{-1}( l_i + r_i ) \).
The convergence is uniform on compact subsets of \( (0, \infty) \).
For \( f \in C_0(S) \), the formula holds also for \( t = 0 \), and the
convergence is uniform on compact subsets of \( [0, \infty) \).
\end{theorem}

The aim of this paper is to prove ``dual'' version of
Theorem~\ref{thm:convergence-in-C}.
Loosely speaking, the main result is as follows (see
Theorems~\ref{thm:dual-is-Markov},
\ref{thm:convergence-of-dual},~\ref{thm:An-gen}, and~\ref{thm:agn-asymptotics}
for precise formulation).

\begin{maintheorem}
For each \( n \in \N \) the part \( A_n^{*} \) of the adjoint operator of
\( A_n \) in the space of Lebesgue integrable \textup{(}or square
integrable\textup{)} functions on \( S \), generates a~strongly continuous
semigroup.
Moreover, the semigroups generated by \( A_n^{*} \)'s converge strongly to
\( \e^{tQ} P \) as \( n \) goes to infinity, for the projection \( P \) given by
the same formula as in \( C(S) \) and some ``matrix'' \( Q \).
\end{maintheorem}

We give an explicit formula for \( Q \), and it is slightly different from
\( Q \) of Theorem~\ref{thm:convergence-in-C}.

One may wish to mimic the proof of the continuous case but this is not fully
possible.
In particular, in the space of continuous functions on \( S \) there exists an
isomorphism transforming boundary conditions associated with the original
process to much simpler homogeneous Neumann boundary conditions.
Because of that, we can obtain limit for the isomorphic semigroups which leads
to required asymptotics.
What is crucial, in the Lebesgue-type space of integrable or square integrable
functions such isomorphism does not exists.
However, there is an isomorphism of the Sobolev space \( W^{2,1} \) or
\( W^{2,2} \) in a~way similar to the isomorphism in the space of continuous
functions.
This leads to a~different approach via Kurtz's convergence
theorem~\cite[Theorem~1.7.6]{MR838085} in \( L^1 \)-type space or, in
\( L^2 \)-type space, Ouhabaz's monotone convergence theorem for sesquilinear
forms~\cite[Theorem~5]{ouhabaz95}, which generalizes Simon's
theorem~\cite[Theorem~3.1]{simon78}.

For generation results in \( C \), \( L^1 \) or \( W^{1,1} \)-type space,
concerning a~diffusion operator with generalized transmission conditions see
also~\cite{banasiak2016}.

\section{Analysis in \texorpdfstring{\( L^1(S) \)}{L1(S)}}
\label{sec:analysis--l1}

We consider a~model that is in a~way dual to that described in
Section~\ref{sec:continuous-case} by investigating the restriction of the
adjoint of \( A_n \) to \( L^1 \)-type space.

In order to set up notations, for an interval \( I \subset \R \) equipped with
the Lebesgue measure, let \( L^1(I) \) be the real space of (equivalence classes
of) Lebesgue integrable real functions defined on \( I \).
By \( \dnorm_{L^1(I)} \) we denote the standard norm
\[
\norm{ \phi }_{L^1(I)} \coloneqq \int_I \abs{ \phi(t) } \dd t = \int_I \abs{
  \phi }, \qquad \phi \in L^1(I).
\]
Moreover, let \( W^{2,1}(S) \) be the Sobolev space of (equivalence classes of)
functions \( \phi \in L^1(I) \) such that \( \phi \) and \( \phi' \) are weakly
differentiable with \( \phi', \phi'' \in L^1(I) \).
The space \( W^{2,1}(I) \) equipped with the norm
\[
\norm{ \phi }_{W^{2,1}(I)} \coloneqq \norm{ \phi }_{L^1(I)} + \norm{ \phi''
}_{L^1(I)}, \qquad \phi \in W^{2,1}(I)
\]
is a~Banach space.
Moreover, if \( (\X,\dnorm_{\X}) \) is a~Banach space, then by
\( \dnorm_{\lin(\X)} \) we denote the operator norm in \( \X \).


\subsection{Adjoint of the operator \texorpdfstring{\( A_n \)}{An}}
\label{sec:adjoint-semigroup}

Using the same identification as in Section~\ref{sec:continuous-case}, we
consider the space
\[
L^1(S) \coloneqq \lrc{\phi\colon \phi = (\phi_i)_{i \in \cn}, \ \phi_i \in
  L^1(E_i)}.
\]
Here \( L^1(E_i) \) is the space of (equivalence classes of) Lebesgue integrable
functions on \( E_i \), identified with \( L^1(0,d_i) \).
More precisely, if \( \ell_i(t) \) is the unique point on the edge \( E_i \),
whose distance from \( L_i \) (along the edge) is \( t \in (0,d_i) \), then the
function \( \phi \in L^1(E_i) \) is identified with
\( \phi \circ \ell_i \in L^1(0,d_i) \).
Such identification is an isometric isomorphism.
In particular, for \( \phi = (\phi_i)_{i \in \cn} \in L^1(S) \), we have
\( \int_{E_i} \phi_i = \int_0^{d_i} \phi_i \circ \ell_i \) and
\( \int_S \phi = \sum_{i \in \cn} \int_{E_i} \phi_i \).
We introduce the norm \( \norm{ \, \cdot \, }_{L^1(S)} \) in \( L^1(S) \) by
\[
\norm{ \phi }_{L^1(S)} \coloneqq \sum_{i \in \cn} \norm{ \phi_i \circ \ell_i
}_{L^1(0,d_i)}, \qquad \phi = (\phi_i)_{i \in \cn} \in L^1(S).
\]
Let \( W^{2,1}(S) \) be the Sobolev-type space on \( S \), that is the subspace
of \( L^1(S) \) composed of (equivalence classes of) functions
\( \phi \in L^1(S) \) such that \( \phi \) and \( \phi' \) are weakly
differentiable and \( \phi', \phi'' \in L^1(S) \).

Let \( \sigma \), \( \seq{\kappa_n}_{n \in \N} \), and all \( l_i \), \( r_i \),
\( l_{ij} \), \( r_{ij} \)'s be as in Section~\ref{sec:continuous-case}.
For each \( n \in \N \) we define the operator \( A_n^{*} \) in \( L^1(S) \) by
\begin{equation}
\label{eq:A-dual}
A_n^{*} \phi = \kappa_n \sigma \phi''
\end{equation}
with domain \( D(A_n^{*}) \) composed of members of \( W^{2,1}(S) \) satisfying
the transmission conditions
\begin{align}
\kappa_n \sigma_i \phi'(L_i) &= \sigma_i l_i \phi(L_i) - \sump_{j \in I_i^L}
\lrb[\big]{ \sigma_j l_{ji} \phi(L_j) + \sigma_j r_{ji} \phi(R_j) },\label{eq:transition-conditions-dual1}\\
\kappa_n \sigma_i \phi'(R_i) &= \sump_{j \in I_i^R} \lrb[\big]{ \sigma_j l_{ji}
  \phi(L_j) + \sigma_j r_{ji} \phi(R_j) } - \sigma_i r_i \phi(R_i),\label{eq:transition-conditions-dual2}
\end{align}
for all \( i \in \cl N \).
Here, \( I_i^L \) and \( I_i^R \) are the sets of indexes \( j \neq i \) of
edges incident in \( L_i \) and \( R_i \), respectively.
The prime in the sums denotes the fact that, since there are no loops, at most
one of the terms \( \sigma_j l_{ji} \phi(L_j) \) and
\( \sigma_j r_{ji} \phi(R_j) \) is taken into account.
Denoting the right-hand sides of~\eqref{eq:transition-conditions-dual1}
and~\eqref{eq:transition-conditions-dual2} by, respectively,
\( \sigma_i F_{L,i} \phi \) and \( \sigma_i F_{R,i} \phi \), we may rewrite
these conditions in the form
\begin{equation}
\label{eq:transmission-cond-F}
\kappa_n \phi'(L_i) = F_{L,i} \phi, \qquad \kappa_n \phi'(R_i) = F_{R,i} \phi,
\qquad i \in \cn,
\end{equation}
and consider \( F_{L,i} \), \( F_{R,i} \) as linear functionals in
\( W^{2,1}(S) \).

Keeping in mind the Riesz representation theorem, the following lemma shows that
the operator \( A_n^{*} \) is in a way adjoint to \( A_n \) defined in
Section~\ref{sec:continuous-case}.
More precisely, \( A_n^{*} \) is the \emph{part} (see~\cite[p.~60]{MR1721989})
of the adjoint of \( A_n \) in \( L^1(S) \).

\begin{lemma}
\label{lem:dual-semigroup}
Let \( n \in \N \).
If \( f \in D(A_n) \) and \( \phi \in D(A_n^{*}) \), then
\begin{equation}
\label{eq:dual-semigroup}
\int_S \phi A_n f = \int_S (A_n^{*} \phi) f.
\end{equation}
\end{lemma}

\begin{proof}
Integrating by parts we obtain
\begin{equation*}
\begin{split}
\int_{E_j} \phi f'' &= \phi(R_j) f'(R_j) - \phi'(R_j) f(R_j) + \phi'(L_j) f(L_j)
- \phi(L_j) f'(L_j) + \int_{E_j} \phi'' f
\end{split}
\end{equation*}
for every \( j \in \cl N \).
Hence, equality~\eqref{eq:dual-semigroup} holds if and only if
\begin{equation}
\label{eq:transmission-conditions-on-edge}
\sum_{j \in \cn} \sigma_j \lrb[\big]{\phi(R_j) f'(R_j) - \phi'(R_j) f(R_j) + \phi'(L_j)
  f(L_j) - \phi(L_j) f'(L_j)} = 0.
\end{equation}
Since \( f \) belongs to \( D(A_n) \), transmission
conditions~\eqref{eq:transmission-condition-1}
and~\eqref{eq:transmission-condition-2}, with left-hand sides multiplied by
\( \kappa_n \), are satisfied.
Thus~\eqref{eq:transmission-conditions-on-edge} holds if and only if
\begin{align*}
  &\kappa_n^{-1} \sum_{j \in \cl N} \sigma_j \phi(R_j) \lrb[\bigg]{ \sum_{i \neq j} r_{ji}
    f(R_{ji}) - r_j f(R_j) } - \sum_{j \in \cl N} \sigma_j \phi'(R_j) f(R_j)\\
  + &\kappa_n^{-1} \sum_{j \in \cl N} \sigma_j \phi(L_j) \lrb[\bigg]{ \sum_{i \neq j} l_{ji}
      f(L_{ji}) - l_j f(L_j) } + \sum_{j \in \cl N} \sigma_j \phi'(L_j) f(L_j) = 0,
\end{align*}
where \( L_{ji} \) and \( R_{ji} \) are, by definition, respectively left and
right ends of \( E_j \), seen as members of \( E_i \).
Changing the order of summation, the last equality becomes
\begin{align*}
  &\kappa_n^{-1} \sum_{i \in \cl N} \sum_{j \neq i} \sigma_j r_{ji} \phi(R_j) f(R_{ji}) - \sum_{i
    \in \cl N} \sigma_i f(R_i) \lrb[\big]{\kappa_n^{-1} r_i \phi(R_i) + \phi'(R_i)}\\
  + &\kappa_n^{-1} \sum_{i \in \cl N} \sum_{j \neq i} \sigma_j l_{ji} \phi(L_j) f(L_{ji}) -
      \sum_{i \in \cl N} \sigma_i f(L_i) \lrb[\big]{ \kappa_n^{-1} l_i \phi(L_i) - \phi'(L_i) } = 0.
\end{align*}
Notice that \( L_{ji} \) is either \( L_i \) or \( R_i \), or is left undefined,
and the same holds for \( R_{ji} \). Thus we can rewrite the last condition in
the form
\[
\sum_{i \in \cl N} \sigma_i f(R_i) \lrb[\big]{ \kappa_n^{-1} F_{R,i} \phi -
  \phi'(R_i) } + \sum_{i \in \cl N} \sigma_i f(L_i) \lrb[\big]{ \phi'(L_i) -
  \kappa_n^{-1} F_{L,i} \phi } = 0,
\]
which is true, since \( \phi \) satisfies the transition
conditions~\eqref{eq:transmission-cond-F}.
\end{proof}

\subsection{One-dimensional Laplacian in
  \texorpdfstring{\( L^1(a,b) \)}{L1(a,b)} and
  \texorpdfstring{\( W^{2,1}(a,b) \)}{W21(a,b)}}
\label{sec:sect-lap}

Let \( a,b \in \R \) be such that \( a < b \), and consider the one-dimensional
Laplacian \( G \) in \( L^1(a,b) \) with homogeneous Neumann boundary
conditions, that is
\[
G\phi \coloneqq \phi''
\]
with the domain \( D(G) \) composed of functions \( \phi \in W^{2,1}(a,b) \)
satisfying
\[
\phi'(a) = \phi'(b) = 0,
\]
It is easy to check that \( D(G) \) is dense in \( L^1(a,b) \).
Moreover, standard calculations show, see e.g.~\cite[Proposition~2.1.2 and
Exercise~2.1.3.4]{lorenzi-lunardi-isem}, that the resolvent set of \( G \)
contains the interval \( (0,+\infty) \), and there exists \( M > 0 \) such that
\begin{equation}
\label{eq:laplacian-res-est}
\norm{ \lambda(\lambda - G)^{-1} }_{\lin(L^1(a,b))} \leq M
\end{equation}
for every \( \lambda > 0 \).
Consequently, by the Hill-Yosida theorem, the operator \( G \) generates
a~strongly continuous semigroup of contractions in \( L^1(a,b) \).

In the following two propositions we also need an explicit formula for the
resolvent of \( G \).
For \( \nu \in \R \) let \( \e_{\nu} \) be the function defined by
\( \e_{\nu}(x) \coloneqq \e^{-\nu x} \) for \( x \in \R \).
Fix \( \phi \in L^1(a,b) \) and \( \lambda > 0 \).
The function \( \psi_{\lambda} \coloneqq (\lambda - G)^{-1} \phi \in D(G) \)
satisfies the resolvent equation
\[
\lambda \psi_{\lambda} - \psi_{\lambda}'' = \phi.
\]
Hence, letting \( \mu \coloneqq \sqrt{\lambda} \), we may write
\( \psi_{\lambda} \) in the form
\begin{equation}
\label{eq:sol-resol}
\psi_{\lambda} = J_{\mu} + c_{\mu} \e_{-\mu} + d_{\mu} \e_{\mu},
\end{equation}
where
\[
J_{\mu}(x) \coloneqq \frac{1}{2\mu} \int_a^b \e_{\mu}(\abs{ x - y }) \phi(y) \dd
y, \qquad x \in (a,b),
\]
and \( c_{\mu} \), \( d_{\mu} \), depending merely on \( \mu \) and \( \phi \),
are chosen so that \( \psi'(a) = \psi'(b) = 0 \).
Precisely,
\begin{equation}
\label{eq:sol-constants}
c_{\mu} = \frac{\xi_{\mu} \e^{-\mu b} + \zeta_{\mu} \e^{-\mu
    a}}{ \mu \lrp[\big]{\e^{\mu(b-a)} - \e^{-\mu(b-a)}} },
\qquad d_{\mu} = \frac{ \xi_{\mu} \e^{\mu b} + \zeta_{\mu} \e^{\mu a} }{
  \mu \lrp[\big]{\e^{\mu(b-a)} - \e^{-\mu(b-a)}} },
\end{equation}
for
\[
\xi_\mu \coloneqq \frac{1}{2} \int_a^b \e_{\mu}(y-a) \phi(y) \dd y, \qquad
\zeta_{\mu} \coloneqq \frac{1}{2} \int_a^b \e_{\mu}(b-y) \phi(y) \dd y.
\]

\begin{proposition}
\label{prop:convergence-res}
For every \( \phi \in L^1(a,b) \) we have
\[
\lim_{\lambda \to 0^+} \lambda(\lambda - G)^{-1} \phi = \frac{1}{b-a} \int_a^b
\phi,
\]
in \( L^1(a,b) \), where \( \int_a^b \phi \) is identified with the constant
function on \( (a,b) \).
\end{proposition}

\begin{proof}
For fixed \( \phi \in L^1(a,b) \) and \( \lambda > 0 \), let
\( \psi_{\lambda} \coloneqq (\lambda - G)^{-1} \phi \) be given
by~\eqref{eq:sol-resol}.
Observe that \( \norm{ J_{\mu} }_{L^1(a,b)} \leq (2\mu)^{-1} \norm{ \phi
}_{L^1(a,b)} \).
Consequently, letting \( C \coloneqq (b-a)^{-1} \int_a^b \phi \),
\[
\norm{ \lambda \psi_{\lambda} - C }_{L^1(a,b)} \leq 2^{-1}\mu \norm{ \phi
}_{L^1(a,b)} + \norm{ \mu^2 c_{\mu} \e_{-\mu} + \mu^2 d_{\mu} \e_{\mu} - C
}_{L^1(a,b)}.
\]
Since
\[
\lim_{\mu \to 0^+} \mu \lrp[\big]{\e^{\mu(b-a)} - \e^{-\mu(b-a)}}^{-1} =
[2(b-a)]^{-1},
\]
it follows that \( \mu^2 c_{\mu} \) and \( \mu^2 d_{\mu} \) both converge to
\( C/2 \) as \( \mu \to 0^+ \), by the Lebesgue dominated convergence theorem.
Using the Lebesgue dominated convergence theorem again, we see that
\[
\lim_{\mu \to 0^+} \norm{ \mu^2 c_{\mu} \e_{-\mu} + \mu^2 d_{\mu} \e_{\mu} - C
}_{L^1(a,b)} = 0.
\]
Therefore
\[
\lim_{\lambda \to 0^+} \norm{ \lambda \psi_{\lambda} - C }_{L^1(a,b)} = 0,
\]
which completes the proof.
\end{proof}
%

Let \( \tiG \) be the part of \( G \) in \( W^{2,1}(a,b) \), that is
\[
\tiG \phi \coloneqq G\phi
\]
with domain
\[
D(\tiG) \coloneqq \set{\phi \in D(G) \cap W^{2,1}(a,b)\colon G\phi \in
  W^{2,1}(a,b)}.
\]

\begin{proposition}
\label{prop:part-res-estimate}
The resolvent set of \( \tiG \) contains the interval \( (0,+\infty) \), and
there exists \( \tilde M > 0 \) such that
\[
\norm{ \lambda(\lambda - \tiG)^{-1} }_{\lin(W^{2,1}(a,b))} \leq \tilde M
\]
for every \( \lambda > 0 \).
\end{proposition}

\begin{proof}
For simplicity, we assume that \( a = 0 \) and \( b = 1 \).
The general case follows in the same way.

For fixed \( \phi \in W^{2,1}(0,1) \) and \( \lambda > 0 \), let
\( \psi_{\lambda} \coloneqq (\lambda - G)^{-1}\phi \) be given
by~\eqref{eq:sol-resol}.
Since \( \lambda \psi_{\lambda} - \psi_{\lambda}'' = \phi \), and
\( \phi \in W^{2,1}(0,1) \), it follows that
\( \psi_{\lambda}'' \in W^{2,1}(0,1) \).
Let \( D_{\mu} \coloneqq 1-\e^{-2\mu} \), and rewrite~\eqref{eq:sol-constants}
in the form
\[
c_{\mu} = \frac{1}{2\mu D_{\mu}} \int_0^1 (\e_{\mu}(2+y) + \e_{\mu}(2-y))
\phi(y) \dd y,
\]
and
\[
d_{\mu} = \frac{1}{2\mu D_{\mu}} \int_0^1 (\e_{\mu}(y) +
\e_{\mu}(2-y)) \phi(y) \dd y.
\]
Then~\eqref{eq:sol-resol} takes the form
\begin{equation*}
\begin{split}
\psi_{\lambda}(x) = \frac{1}{2\mu D_{\mu}} \int_0^1 \lrb[\big]{ &\e_{\mu}(\abs{
    x-y }) - \e_{\mu}(2+\abs{ x-y }) + \e_{\mu}(2-x+y)\\ &+ \e_{\mu}(2-x-y) +
  \e_{\mu}(x+y) + \e_{\mu}(2+x-y)} \phi(y) \dd y
\end{split}
\end{equation*}
for \( x \in (0,1) \).
Expanding \( D_{\mu}^{-1} \) into the geometric series
\( \sum_{k \geq 0} \e^{-2\mu k} \) we note that
\[
D_{\mu}^{-1} \lrb[\big]{ \e_{\mu}( t ) - \e_{\mu}(2+t) } = \e_{\mu}( t ), \qquad
t \geq 0,
\]
and
\[
D_{\mu}^{-1} \lrb[\big]{ \e_{\mu}(2-t) + \e_{\mu}(t) } = \sum_{k \geq
  1} \e_{\mu}(2k-t) + \sum_{k \leq 0} \e_{\mu}(-2k+t) = \sum_{k \in \Z }
\e_{\mu}(\abs{ 2k+t }),
\]
for \( t \in [0,2] \).
Similarly,
\begin{equation*}
\begin{split}
D_{\mu}^{-1} \lrb[\big]{ \e_{\mu}(2-t) + \e_{\mu}(2+t) } = \sum_{k \in \Z }
\e_{\mu}(\abs{ 2k+t }) - \e_{\mu}(\abs{ t }), \qquad t \in [-2,2].
\end{split}
\end{equation*}
Finally, we may rewrite \( \psi_{\lambda} \) in the form
\[
\frac{1}{2\mu} \sum_{k \in \Z} \int_0^1 \lrb[\big]{ \e_{\mu}(\abs{
    2k+x-y }) + \e_{\mu}(\abs{ 2k+x+y }) } \phi(y) \dd y, \qquad x \in (0,1).
\]
Changing variables in the integral we see that
\begin{equation}
\label{eq:sol-convolution}
\psi_{\lambda}(x) = \frac{1}{2\mu} \sum_{k \in \Z} \lrb[\bigg]{ \int_{x-1}^x \e_{\mu}(\abs{
    2k+t }) \phi(x-t) \dd t + \int_x^{x+1}
  \e_{\mu}(\abs{ 2k+t }) \phi(t-x) \dd t }
\end{equation}
for \( x \in (0,1) \).
Differentiating this twice leads to
\[
\psi_{\lambda}'' = \Psi_{\lambda} + \Phi_{\lambda},
\]
where \( \Psi_{\lambda} \) is given by the right-hand side
of~\eqref{eq:sol-convolution} with \( \phi \) replaced by \( \phi'' \), that is
\( \Psi_{\lambda} = (\lambda - G)^{-1} \phi'' \), and
\[
\Phi_{\lambda}(x) \coloneqq \frac{1}{\mu} \sum_{k \in \Z} \lrb[\big]{
  \e_{\mu}(\abs{ 2k+x }) \phi'(0) - \e_{\mu}(\abs{ 2k-1+x }) \phi'(1) }, \qquad
x \in (0,1).
\]
In order to estimate the norm of \( \Phi_{\lambda} \), observe that
\[
\sum_{k \geq 1} \e_{\mu}(2k) = \frac{1}{\e^{2\mu}-1} < \frac{1}{2\mu},
\]
which implies
\[
\sum_{k \in \Z} \e_{\mu}(\abs{ 2k + t }) \leq \e_{\mu}(\abs{ t }) + 2 \sum_{k
  \geq 1} \e_{\mu}(2k) = \e_{\mu}(\abs{ t }) + \frac{1}{\mu}, \qquad t \in
[-1,1].
\]
Therefore, by the Sobolev embedding theorem,
\[
\norm{ \lambda \Phi_{\lambda} }_{L^1(0,1)} \leq 2\lrp[\big]{\mu \norm{ \e_{\mu}
  }_{L^1(0,1)} + 1 } \norm{ \phi }_{W^{2,1}(0,1)} = 4 \norm{ \phi
}_{W^{2,1}(0,1)}.
\]
Finally, by~\eqref{eq:laplacian-res-est},
\begin{equation*}
\begin{split}
\norm{ \lambda \psi_{\lambda} }_{W^{2,1}(0,1)} &\leq \norm{ \lambda
  \psi_{\lambda} }_{L^1(0,1)} + \norm{ \lambda \Psi_{\lambda} }_{L^1(0,1)} +
\norm{ \lambda \Phi_{\lambda} }_{L^1(0,1)}\\ &\leq M \norm{ \phi }_{L^1(0,1)} +
M \norm{ \phi'' }_{L^1(0,1)} + 4 \norm{ \phi }_{W^{2,1}(0,1)}\\
&\leq (2M + 4) \norm{ \phi }_{W^{2,1}(0,1)},
\end{split}
\end{equation*}
which completes the proof.
\end{proof}

\subsection{Generation theorem in \texorpdfstring{\( L^1(S) \)}{L1(S)}}
\label{sec:sub-sub-Markov-semigroup}

As we said before, we know from~\cite[Proposition~2.1]{Bobrowski-diff} that for
each \( n \in \N \) the operator \( A_n \) generates a~Feller semigroup in
\( C(S) \).
We prove that the operator \( A_n^{*} \) defined in
Section~\ref{sec:adjoint-semigroup} generates a \emph{sub-Markov} semigroup
\( \set{ \e^{tA_n^{*}} }_{t \geq 0} \) in \( L^1(S) \), that is a~semigroup of
operators such that for every nonnegative \( \phi \in L^1(S) \) we have
\( \e^{tA_n^{*}} \phi \geq 0 \) and
\( \int_S \e^{tA_n^{*}} \phi \leq \int_S \phi \) for all \( t \geq 0 \).
Moreover, if the semigroup \( \sem{ \e^{tA_n} }_{t \geq 0} \) is conservative,
we show that \( \sem{ \e^{tA_n^{*}} }_{t \geq 0} \) is \emph{Markov}, that is
\( \int_S \e^{tA_n^{*}} \phi = \int_S \phi \) for all nonnegative
\( \phi \in L^1(S) \).
The main theorem of this section is as follows.

\begin{theorem}
\label{thm:dual-is-Markov}
For each \( n \in \N \) the operator \( A_n^{*} \) generates a sub-Markov
semigroup in \( L^1(S) \).
Moreover, if the semigroup generated by \( A_n \) is conservative, then
\( A_n^{*} \) generates a Markov semigroup.
\end{theorem}

Before we prove the theorem, we need auxiliary results.
In what follows in this section we fix \( n \in \N \).

\begin{lemma}
\label{lem:existence-of-dual-resolvent}
The resolvent set of \( A_n^{*} \) contains the interval \( (0,+\infty) \).
Moreover, for each \( \lambda > 0 \) we have
\[
\norm{ \lambda(\lambda - A_n^{*})^{-1} }_{\lin(L^1(S))} \leq 1
\]
and
\begin{equation}
\label{eq:dual-resolvent}
\int_S \phi (\lambda - A_n)^{-1} f = \int_S f (\lambda - A_n^{*})^{-1} \phi, \qquad
f \in C(S), \ \phi \in L^1(S).
\end{equation}
\end{lemma}

Observe that if the resolvent of \( A_n^{*} \) exists, then
equality~\eqref{eq:dual-resolvent} becomes obvious
(see~\cite[Lemma~1.10.2]{MR710486}), since \( A_n^{*} \) is the part of the
adjoint of \( A_n \) in \( L^1(S) \).
However, the existence of \( (\lambda - A_n^{*})^{-1} \) is not obvious.
We prove Lemma~\ref{lem:existence-of-dual-resolvent} in a~moment.
First we show that the part of \( \lambda - A_n^{*} \) in \( W^{2,1}(S) \)
satisfies the range condition.

\begin{lemma}
\label{lem:range-cond}
For sufficiently large \( \lambda > 0 \) the image of\/ \( D(A_n^{*}) \) under
the operator \( \lambda - A_n^{*} \) contains \( W^{2,1}(S) \).
\end{lemma}

Before proving the lemma we introduce some notations.
It is crucial in our analysis to consider \( L^1(S) \) with the Bielecki-type
norm, see~\cite{bielecki56} or~\cite[p.~56]{granas2003}.
For \( i \in \cn \) and \( \omega > 0 \) let \( \norm{ \, \cdot \, }_{\omega} \)
be the norm in \( L^1(0,d_i) \) given by
\[
\norm{ \phi }_{\omega} \coloneqq \sup_{t \in (0,d_i)} \e^{-\omega t} \int_0^t
\abs{ \phi(s) } \dd s, \qquad \phi \in L^1(0,d_i).
\]
Naturally, see the beginning of Section~\ref{sec:adjoint-semigroup}, we set
\[
\norm{ \phi }_{\omega} \coloneqq \sum_{i \in \cn} \norm{ \phi_i \circ e_i
}_{\omega}, \qquad \phi = (\phi_i)_{i \in \cn} \in L^1(S).
\]
Such norm in \( L^1(S) \) is equivalent to the standard norm
\( \norm{ \, \cdot \, }_{L^1(S)} \).
It is also clear that
\begin{equation}
\label{eq:bielecki-lim}
\lim_{\omega \to +\infty} \norm{ \phi }_{\omega} = 0
\end{equation}
for every \( \phi \in L^1(S) \).
Furthermore, by \( \normt{ \, \cdot \, }_{\omega} \) we denote the related norm
in \( W^{2,1}(S) \), that is
\[
\normt{ \phi }_{\omega} \coloneqq \norm{ \phi }_{\omega} + \norm{ \phi''
}_{\omega}, \qquad \phi \in W^{2,1}(S).
\]
Finally, for simplicity of notation, also by \( \norm{ \, \cdot \, }_{\omega} \)
and \( \normt{ \, \cdot \, }_{\omega} \) we denote the operator norms
corresponding to defined above Bielecki's norms in, respectively, \( L^1(S) \)
and \( W^{2,1}(S) \).

For \( i \in \cn \) let \( G_i \) be the version of \( G \), see
Section~\ref{sec:sect-lap}, in \( L^1(E_i) \), and let \( B \) be the operator
in \( L^1(S) \) defined by
\begin{equation}
\label{eq:B-def}
B \phi \coloneqq \sigma \phi''
\end{equation}
with the domain \( D(B) \) composed of functions
\( \phi = (\phi_i)_{i \in \cn} \) such that \( \phi_i \in D(G_i) \).
Since \( B \) is equal to \( G_i \) on each \( E_i \),
operator \( B \) generates strongly continuous semigroup
\( \set{ \e^{tB} }_{t \geq 0} \).

The main idea of the proof of Lemma~\ref{lem:range-cond} is to consider the
isomorphic image of the part of \( A_n^{*} \) in \( W^{2,1}(S) \).
It turns out that it is possible to choose an isomorphism such that the image is
a~perturbation of the Laplacian with homogeneous Neumann boundary conditions.

For each \( i \in \cn \) choose smooth functions \( h_{L,i} \), \( h_{R,i} \)
defined on \( E_i \), and such that
\[
h_{L,i}(L_i) = h_{L,i}(R_i) = h_{R,i}(L_i) = h_{R,i}(R_i) = 0,
\]
and
\begin{equation}
\label{eq:h-derivative}
h_{L,i}'(L_i) = h_{R,i}'(R_i) = 1, \qquad h_{L,i}'(R_i) = h_{R,i}'(L_i) = 0.
\end{equation}
Let \( J \) be the linear operator in \( W^{2,1}(S) \) given by
\[
J \phi = \lrp[\big]{ (F_{L,i} \phi) h_{L,i} + (F_{R,i} \phi) h_{R,i} }_{i \in
  \cn}, \qquad \phi \in W^{2,1}(S),
\]
where \( F_{L,i} \) and \( F_{R,i} \) are linear functionals in \( W^{2,1}(S) \)
defined in Section~\ref{sec:adjoint-semigroup}.
By the Sobolev embedding theorem the functionals are bounded and there exists
\( M > 0 \), depending merely on permeability coefficients, such that
\[
\max_{i \in \cn} \normt{ F_{L,i} \phi + F_{R,i} \phi }_{\omega} \leq M \normt{
  \phi }_{\omega}, \qquad \phi \in W^{2,1}(S),\ \omega > 0.
\]
Hence the operator \( J \) is bounded and we estimate its Bielecki's norm,
obtaining
\begin{equation}
\label{eq:J-bound}
\normt{ J }_{\omega} \leq M \sum_{i \in \cn} (\normt{ h_{L,i} }_{\omega} + \normt{
  h_{R,i} }_{\omega})
\end{equation}
for all \( \omega > 0 \).
Let \( I_n\colon W^{2,1}(S) \to W^{2,1}(S) \) be the bounded linear operator
given by
\begin{equation}
\label{eq:isomorphism-on-W21}
I_n \coloneqq I_{W^{2,1}(S)} - \kappa_n^{-1} J.
\end{equation}
Here, \( I_{W^{2,1}(S)} \) is the identity operator in \( W^{2,1}(S) \).
Then, the choice of \( h_{L,i} \), \( h_{R,i} \) guarantees
(see~\cite[Lemma~3.1]{Bobrowski-diff}) that \( I_n \) is an isomorphism of
\( W^{2,1}(S) \) with the inverse
\begin{equation}
\label{eq:inverse-of-isomorphism}
I_n^{-1} = I_{W^{2,1}(S)} + \kappa_n^{-1} J.
\end{equation}
What is crucial, note that
\begin{equation}
\label{eq:domain-isomorphism}
\phi \in D(A_n^{*}) \quad \text{ if and only if } \quad I_n \phi
\in D(B),
\end{equation}
or equivalently, \( I_n \) is an isomorphism between functions in
\( W^{2,1}(S) \) satisfying conditions~\eqref{eq:transmission-cond-F}, and those
satisfying homogeneous Nuemann boundary conditions on each edge.

Furthermore, let \( K\colon W^{2,1}(S) \to W^{2,1}(S) \) be defined as \( J \)
with \( h_{L,i} \), \( h_{R,i} \) replaced by their second derivatives, that is
\begin{equation}
\label{eq:K-def}
K \phi \coloneqq \lrp[\big]{ (F_{L,i} \phi) h_{L,i}'' + (F_{R,i} \phi) h_{R,i}''
}_{i \in \cn}, \qquad \phi \in W^{2,1}(S).
\end{equation}

\begin{proof}[Proof of Lemma~\textup{\ref{lem:range-cond}}]
We consider \( W^{2,1}(S) \) as a~Banach space with the Bielecki norm
\( \normt{ \, \cdot \, }_{\omega} \).
We define \( \tilde B_n \coloneqq I_{n} \tiA_n^{*} I_{n}^{-1} \), where
\( \tiA_n^{*} \) is the part of \( A_n^{*} \) in \( W^{2,1}(S) \).
We have
\[
\tiA_n^{*} I_n^{-1} = \kappa_n \tiB + \sigma K,
\]
where \( \tiB \) is the part of \( B \) in \( W^{2,1}(S) \).
Moreover,
\[
D(\tiB_n) = I_n D(\tiA_n^{*}) = D(\tiB),
\]
where the last equality is a~consequence of~\eqref{eq:domain-isomorphism}.
Furthermore, by~\eqref{eq:isomorphism-on-W21},
\[
\tiB_n = \kappa_n \tiB - J\tiB + \sigma K - \kappa_n^{-1} \sigma JK.
\]
Denoting \( C \coloneqq -J \tiB \),
\( D \coloneqq \sigma K - \kappa_n^{-1} \sigma JK \), we have
\begin{equation}
\label{eq:B-partition}
\tiB_n = \kappa_n \tiB + C + D.
\end{equation}

By Proposition~\ref{prop:part-res-estimate}, there exists \( \tiM > 0 \) such
that \( \normt{ \lambda (\lambda - \kappa_n \tiB)^{-1} }_{\omega} \leq \tiM \)
for all \( \lambda > 0 \) (recall that the standard norm and the Bielecki norm
are equivalent).
Hence, using the fact that
\[
\kappa_n\tiB (\lambda - \kappa_n\tiB )^{-1} = \lambda(\lambda - \kappa_n
\tiB)^{-1} - I_{W^{2,1}(S)},
\]
we have
\begin{equation*}
\normt{ C(\lambda - \kappa_n \tiB)^{-1} }_{\omega} \leq \kappa_n^{-1} \normt{ J
}_{\omega} \normt{ \lambda(\lambda - \tiB)^{-1} - I_{W^{2,1}(S)} }_{\omega} \leq
\kappa_n^{-1} \normt{ J }_{\omega} (\tiM + 1).
\end{equation*}
Choose \( \omega > 0 \) large enough, so that
\( \normt{ J }_{\omega} < \kappa_n (\tiM+1)^{-1} \).
Such \( \omega \) exists by~\eqref{eq:bielecki-lim} and~\eqref{eq:J-bound}.
Then
\[
q \coloneqq \normt{ C(\lambda - \kappa_n \tiB)^{-1} }_{\omega} < 1,
\]
which implies that \( \lambda \in \rho(\kappa_n\tiB + C) \).
Therefore we have Neumann series expansion
\[
(\lambda - \kappa_n \tiB - C)^{-1} = (\lambda - \kappa_n \tiB)^{-1} \sum_{k \geq
  0} [C(\lambda - \kappa_n \tiB)^{-1}]^k,
\]
and consequently
\[
\normt{ (\lambda - \kappa_n \tiB - C)^{-1} }_{\omega} \leq \frac{\tiM}{1-q}
\frac{1}{\lambda}.
\]
This means that \( \kappa_n \tiB + C \), being densely defined, generates
a~strongly continuous semigroup in \( W^{2,1}(S) \).
What is more, the operator \( D \) in~\eqref{eq:B-partition} is bounded, since
\( J \) and \( K \) are.
Hence, by the bounded perturbation theorem (see
e.g.~\cite[Proposition~III.1.12]{MR1721989}), the operator \( \tiB_n \)
generates a~strongly continuous semigroup in \( W^{2,1}(S) \), and so does its
isomorphic image \( \tiA_n^{*} \).
In particular \( (\lambda - \tiA_n^{*})(D(\tiA_n^{*})) = W^{2,1}(S) \) for
sufficiently large \( \lambda > 0 \).
\end{proof}

\begin{remark}
\label{rem:domain-dense}
We showed in the proof of Lemma~\ref{lem:range-cond} that \( \tiA_n^{*} \) is
the generator of a~strongly continuous semigroup in \( W^{2,1}(S) \).
In particular, the domain \( D(\tiA_n^{*}) \) of \( \tiA_n^{*} \) is dense in
\( W^{2,1}(S) \) equipped with the norm \( \normt{ \, \cdot \, }_{\omega} \).
The norm is stronger than the \( L^1 \)-type norm in \( W^{2,1}(S) \).
Therefore, since \( W^{2,1}(S) \) is dense in \( L^1(S) \), the domain of
\( A_n^{*} \), which contains \( D(\tiA_n^{*}) \), is dense in \( L^1(S) \).
\end{remark}

We are now ready to show that the resolvent of \( A_n^{*} \) exists, as claimed in
Lemma~\ref{lem:existence-of-dual-resolvent}.

\begin{proof}[Proof of Lemma~\textup{\ref{lem:existence-of-dual-resolvent}}]

By the remark stated after the lemma, it is enough to show that
\( \lambda - A_n^{*} \) is invertible, and that the norm of the inverse is
bounded by \( \lambda^{-1} \).

First we show that the operator \( A_n^{*} \) is dissipative, that is
\begin{equation}
\label{eq:dissipative}
\norm{ (\lambda - A_n^{*})\phi }_{L^1(S)} \geq \lambda \norm{ \phi }_{L^1(S)},
\qquad \phi \in D(A_n^{*})
\end{equation}
for all \( \lambda > 0 \).
Let \( A_n' \) be the adjoint operator of \( A_n \) in the dual space of
\( C(S) \).
That is \( A_n' \) acts in the space \( M_b(S) \) of regular Borel measures on
\( S \).
As we said in Section~\ref{sec:continuous-case}, \( A_n \) generates a~Feller
semigroup in \( C(S) \), hence \( \norm{ \e^{tA_n} }_{\lin(C(S))} \leq 1 \) for
all \( t \geq 0 \).
Therefore \( \norm{ (\lambda - A_n)^{-1} }_{\lin(C(S))} \leq \lambda^{-1} \) for
all \( \lambda > 0 \).
However, we know (see e.g.~\cite[Theorem~1.10.2]{MR710486}) that for each
\( \lambda \in \rho(A_n) \) it follows that \( \lambda \in \rho(A_n^{*}) \) and
the adjoint of \( (\lambda - A_n)^{-1} \) equals \( (\lambda - A_n')^{-1} \).
Consequently, since the norm of an operator is the same as the norm of its
adjoint,
\begin{equation}
\label{eq:res-estimate}
\norm{ (\lambda - A_n')^{-1} }_{\lin(M_b(S))} \leq \lambda^{-1}.
\end{equation}
Thus \( \norm{ (\lambda - A_n') \mu }_{M_b(S)} \geq \lambda \norm{ \mu
}_{M_b(S)} \) for all \( \mu \in D(A_n') \subset M_b(S) \).
Let \( \phi \in D(A_n^{*}) \) and denote by \( \mu_{\phi} \in M_b(S) \) the
measure corresponding to \( \phi \), that is the measure defined by
\( \mu_{\phi}(E) \coloneqq \int_E \phi \) for any Borel measurable set
\( E \subset S \).
We have
\[
(A_n' \mu_{\phi})f = \int_S A_n f \dd \mu_{\phi} = \int_S \phi A_n f = \int_S f
A_n^{*} \phi, \qquad f \in D(A),
\]
where in the last equality we used Lemma~\ref{lem:dual-semigroup}.
Hence we may write, with slight abuse of notation,
\( A_n' \mu_{\phi} = A_n^{*} \phi \).
This means that \( (\lambda - A_n')\mu_{\phi} = (\lambda - A_n^{*}) \phi \) for
all \( \phi \in D(A_n^{*}) \), and~\eqref{eq:dissipative} follows
by~\eqref{eq:res-estimate}.

Since \( A_n^{*} \) is dissipative, we are left with proving that
\( \lambda - A_n^{*} \) is surjective for some (hence all) \( \lambda > 0 \).
Since \( A_n' \) is closed and \( L^1(S) \) is a~closed subspace of
\( M_b(S) \), the operator \( A_n^{*} \) is also closed.
Hence, see e.g.~\cite[Proposition~II.3.14(iii)]{MR1721989}, the range of
\( \lambda - A_n^{*} \) is closed in \( L^1(S) \).
However, by Lemma~\ref{lem:range-cond}, for sufficiently large \( \lambda > 0 \)
the range contains \( W^{2,1}(S) \), which is dense in \( L^1(S) \).
Hence the range equals \( L^1(S) \).
\end{proof}



\begin{proof}[Proof of Theorem~\textup{\ref{thm:dual-is-Markov}}]
The domain of \( A_n^{*} \) is dense in \( L^1(S) \) (see
Remark~\ref{rem:domain-dense}), hence by
Lemma~\ref{lem:existence-of-dual-resolvent} it follows that \( A_n^{*} \) is the
generator of a~strongly continuous semigroup in \( L^1(S) \).

It is well known, see e.g.~\cite[Corollary~7.8.1]{lasota94}, that
\( \sem{\e^{tA_n^{*}}}_{t \geq 0} \) is sub-Markov, provided that the operator
\( \lambda(\lambda - A_n^{*})^{-1} \) is sub-Markov for all \( \lambda > 0 \).

We prove that if \( \phi \in L^1(S) \) and \( \phi \geq 0 \), then
\( (\lambda - A_n^{*})^{-1} \phi \geq 0\) for every \( \lambda > 0 \).
Let \( m \) be the Lebesgue measure on \( S \), and suppose, contrary to our
claim, that there exists a function \( \phi \geq 0 \), a set
\( \Gamma \subset S \) with \( m(\Gamma) > 0 \), and a real number
\( \delta > 0 \) such that for some \( \lambda_0 > 0 \) we have
\( (\lambda_0 - A_n^{*})^{-1} \phi \leq -\delta \) almost everywhere on
\( \Gamma \).
Without loss of generality, we may assume that \( \Gamma \) is a~subset of some
edge \( E_i \).
Then, for a given \( \epsilon > 0 \), we choose an open set \( G \subset E_i \)
and a~closed set \( \Gamma' \) such that \( \Gamma' \subset \Gamma \subset G \)
and \( m(G \setminus \Gamma') < \epsilon \).
By the Urysohn lemma, there exists a continuous real function
\( 0 \leq f \leq 1 \) with \( f \equiv 1 \) on \( \Gamma' \) and
\( f \equiv 0 \) outside \( G \).
Then
\begin{align*}
  \int_S f(\lambda_0 - A_n^{*})^{-1}\phi &= \int_{\Gamma'} f(\lambda_0 -
                                           A_n^{*})^{-1}\phi + \int_{G \setminus \Gamma'} f(\lambda_0 - A_n^{*})^{-1}\phi\\
                                         &\leq -\delta m(\Gamma') + \int_{G \setminus \Gamma'} f(\lambda_0 - A_n^{*})^{-1}\phi.
\end{align*}
Since \( \epsilon \) is arbitrary small, it follows that the left-hand side is
strictly negative.
However, by Lemma~\ref{lem:existence-of-dual-resolvent},
\[
\int_S f(\lambda_0 - A_n^{*})^{-1}\phi = \int_S \phi (\lambda_0 - A_n)^{-1} f
\geq 0,
\]
where the inequality is a~consequence of the fact that \( A_n \) generates
a~Feller semigroup.
This leads to contradiction and proves that \( (\lambda - A_n^{*})^{-1} \) is
a~positive operator for each \( \lambda > 0 \).

In order to prove the sub-Markov property, let \( \phi \in L^1(S) \).
Since \( A_n \) generates a Feller semigroup, we have
\begin{equation}
\label{eq:Feller-resolvent-bound}
(\lambda - A_n)^{-1} \mathbbm{1}_S = \int_0^{\infty} \e^{-\lambda t} \e^{t A_n}
\mathbbm{1}_S \dd t \leq \mathbbm{1}_S \int_0^{\infty} \e^{-\lambda t} \dd t =
\lambda^{-1} \mathbbm{1}_S, \qquad \lambda > 0,
\end{equation}
where \( \mathbbm{1}_S \equiv 1 \) on \( S \).
Thus, by Lemma~\ref{lem:existence-of-dual-resolvent},
\[
\int_S \lambda (\lambda - A_n^{*})^{-1} \phi = \int_S \phi \lambda (\lambda -
A_n)^{-1} \mathbbm{1}_S \leq \int_S \phi, \qquad \lambda > 0,\ \phi \in L^1(S),
\]
which completes the first part of the proof.

If we assume that the semigroup generated by \( A_n \) is conservative, then
inequality in~\eqref{eq:Feller-resolvent-bound} becomes equality, and
\( \int_S \lambda (\lambda - A_n^{*})^{-1} \phi = \int_S \phi \) for all
\( \lambda > 0 \) and \( \phi \in L^1(S) \).
\end{proof}

\subsection{Convergence in \texorpdfstring{\( L^1(S) \)}{L1(S)}}
\label{sec:convergence}

To prove a~convergence result that resembles Theorem~\ref{thm:convergence-in-C},
we begin with a~theorem due to Kurtz (see~\cite[Theorem~7.6]{MR838085}
or~\cite[Theorem~42.2]{bobrowski-convergence}).

For each \( n \in \N \) let \( \cl A_n \) be the generator of a~strongly
continuous semigroup \( \sem{ \e^{t \cl A_n} }_{ t \geq 0 } \) in a~Banach space
\( \X \).
Assume that the semigroups are equibounded, that is
\[
\norm{ \e^{t \cl A_n} }_{\lin(\X)} \leq C, \qquad n \in \N,\ t \geq 0
\]
for some \( C > 0 \).
Denote by \( \cAex \) the \emph{extended limit} of \( \seq{\ca_n}_{n \in \N} \),
that is the multivalued operator in \( \X \) with the domain \( D(\cAex) \)
composed of all \( x \in \X \) such that there exists a~sequence
\( \seq{x_n}_{n \in \N} \) in \( \X \) that converges to \( x \) while the limit
of \( \ca_n x_n \) exists as \( n \to +\infty \).
By \( (x,y) \in \cAex \) we mean that \( x \in D(\cAex) \) and
\( \lim_{n \to +\infty} \ca_n x_n = y \) for some sequence
\( \seq{x_n}_{n \in \N} \) in \( D(\cAex) \) converging to \( x \).
Moreover, assume that \( \seq{\epsilon_n}_{n \in \N} \) is a sequence of
positive real numbers converging to \( 0 \), and denote by \( \cBex \) the
extended limit of \( \seq{\epsilon_n \ca_n}_{n \in \N} \).

Suppose also that an operator \( \cB \) with domain \( D(\cB) \) generates
a~strongly continuous semigroup \( \sem{\e^{t \cB}}_{t \geq 0} \) in \( \X \)
such that \( \norm{ \e^{t\cB} }_{\lin(\X)} \leq C \), and that for every
\( x \in \X \) the limit
\begin{equation}
\label{eq:Q-limit}
\lim_{\lambda \to 0^+} \lambda (\lambda - \cB)^{-1} x \eqqcolon \cP x
\end{equation}
exists.
The operator \( \cP \) is a~bounded projection, hence its range, which we denote
by
\[
\Y \coloneqq \range P,
\]
is a~closed subspace of \( \X \).
With this setup we use a~special case of Kurtz's theorem (for a~general version
see~\cite[Theorem~7.6]{MR838085}).

\begin{theorem}
\label{thm:Kurtz}
Let \( \mathcal{A} \) be an operator in \/\( \X \) such that \/\( \Y \) is
a~subset of its domain.
Assume that
\begin{enumerate}
  \item\label{item:Kurtz1} if \( x \in \Y \), then \( (x, \ca x) \in \cAex \),
  \item\label{item:Kurtz2} if \( y \in D(\cB) \), then
\( (y, \cB y) \in \cBex \),
  \item\label{item:Kurtz3} the operator \( \cP \ca \) with domain \( \Y \)
generates a~strongly continuous semigroup in \/\( \Y \).
\end{enumerate}
Then for every \( x \in \X \) and \( t > 0 \),
\[
\lim_{n \to +\infty} \e^{t\cl A_n} x = \e^{t \cP \ca} \cP x
\]
in \( \X \), and the convergence is uniform on compact subsets of
\( (0,\infty) \).
If \( x \in \Y \), then the formula holds also for \( t = 0 \), and the
convergence is uniform on compact subsets of \( [0,+\infty) \).
\end{theorem}

In order to verify conditions~\ref{item:Kurtz1}-\ref{item:Kurtz3} of Kurtz's
theorem we need some lemmas.
Recall that for each \( n \in \N \) the operator \( A_n^{*} \) defined
by~\eqref{eq:A-dual} with transmission
conditions~\eqref{eq:transmission-cond-F}, generates a~strongly continuous
semigroup in \( L^1(S) \).
By \( \Aex \) we denote the extended limit of \( \seq{A_n^{*}}_{n \in \N} \).
Moreover, for \( B \) defined by~\eqref{eq:B-def}, it follows from
Proposition~\ref{prop:convergence-res} that the limit
\[
\lim_{\lambda \to 0^+} \lambda(\lambda - B)^{-1} \phi \coloneqq P\phi
\]
exists for every \( \phi \in L^1(S) \), and that
\begin{equation}
\label{eq:P-formula}
P\phi = \lrp[\Big]{ d_i^{-1} \int_{E_i} \phi_i }_{i \in \cn}, \qquad \phi =
(\phi_i)_{i \in \cn} \in L^1(S).
\end{equation}
The range of \( P \) is the closed subspace of \( L^1(S) \) consisting of all
functions that are constant on each edge.
We denote this subspace by \( L_0^1(S) \), and note that it is isometrically
isomorphic to \( \R^N \) equipped with the appropriate norm.

\begin{lemma}
\label{lem:condition-1-Kurtz}
The domain \( D(\Aex) \) contains \( L_0^1(S) \), and for the operator \( K \)
defined by~\eqref{eq:K-def} we have
\[
(\phi,\sigma K \phi) \in \Aex, \qquad \phi \in L_0^1(S).
\]
\end{lemma}

\begin{proof}
Fix \( \phi \in L_0^1(S) \) and set (see~\eqref{eq:inverse-of-isomorphism})
\[
\phi_n \coloneqq I_{n}^{-1} \phi = \phi + \kappa_n^{-1} J\phi.
\]
Since the operator \( J \) is bounded in \( W^{2,1}(S) \) and
\( \kappa_n^{-1} \to 0 \) as \( n \to +\infty \), the sequence
\( \seq{\phi_n}_{n \in \N} \) converges to \( \phi \) in \( L^1(S) \) as
\( n \to +\infty \).
What is more \( \phi_n \in D(A_n^{*}) \) for every \( n \in \N \)
by~\eqref{eq:domain-isomorphism}, and
\[
A_n^{*} \phi_n = \kappa_n \sigma \phi'' + \sigma (J \phi)'' = \sigma K \phi.
\]
Hence \( (\phi,\sigma K \phi) \in \Aex \), which completes the proof.
\end{proof}

For the next lemma let \( \Bex \) be the extended limit of
\( \seq{\kappa_n^{-1} A_n^{*}}_{n \in \N} \).

\begin{lemma}
\label{lem:condition-2-Kurtz}
For every \( \phi \in D(B) \) we have
\[
(\phi, B \phi) \in \Bex.
\]
\end{lemma}

\begin{proof}
Let \( \phi \in D(B) \) and set \( \phi_n \coloneqq I_n^{-1} \phi \).
Then \( \phi_n \in D(A_n^{*}) \) for all \( n \in \N \)
by~\eqref{eq:domain-isomorphism}.
As in the previous lemma \( \seq{\phi_n}_{n \in \N} \) converges to \( \phi \)
in \( L^1(S) \) as \( n \to +\infty \), and
\[
\kappa_n^{-1} A_n^{*} \phi_n = \sigma \phi'' + \kappa_n^{-1} \sigma (J\phi)'' =
B\phi + \kappa_n^{-1} \sigma (J\phi)''.
\]
Since \( \lim_{n \to +\infty} \kappa_n^{-1} = 0 \) and since \( J \) is bounded,
it follows that \( (\phi,B\phi) \in \Bex \), as claimed.
\end{proof}

We are now ready to apply Theorem~\ref{thm:Kurtz}.
In \( L_0^1(S) \) we define the operator \( Q \) by
\[
Q \phi \coloneqq \sigma PK \phi, \qquad \phi \in L_0^1(S),
\]
where \( K \) is given by~\eqref{eq:K-def}.
Observe that for all \( \phi = (\phi_i)_{i \in \cn} \in W^{2,1}(S) \), we have
\[ 
PK \phi = \lrp[\Big]{ d_i^{-1} \int_{E_i} \lrp[\big]{ (F_{L,i} \phi) h_{L,i}'' +
    (F_{R,i} \phi) h_{R,i}'' } }_{i \in \cn} = (d_i^{-1} F_{R,i} \phi - d_i^{-1}
F_{L,i} \phi)_{i \in \cn}.
\]
The last equality follows by~\eqref{eq:h-derivative}.
Denoting by \( I_i^E \) the set of indexes \( j \neq i \) of edges incident to
\( E_i \), it follows
by~\eqref{eq:transition-conditions-dual1}--\eqref{eq:transmission-cond-F}, that
for every \( \phi \in L_0^1(S) \) we have
\begin{equation}
\label{eq:sigmaPK}
\begin{split}
Q \phi &= \lrp[\Big]{ d_i^{-1} \sum_{j \in I_i^E} \lrb[\big]{ \sigma_j l_{ji}
    \phi(L_j) + \sigma_j r_{ji} \phi(R_j) } - \sigma_i d_i^{-1} \lrb[\big]{ l_i
    \phi(L_i) + r_i \phi(R_i) } }_{i \in \cn}\\
&= \lrp[\Big]{ d_i^{-1} \sum_{j \neq i} \sigma_j (l_{ji} + r_{ji}) \phi_j -
  \sigma_i d_i^{-1} (l_i + r_i) \phi_i }_{i \in \cn},
\end{split}
\end{equation}
where \( \phi_j \) is the value of \( \phi \) on the edge \( E_j \).
We introduce the matrix \( (q_{ij})_{i,j \in \cn} \) by
\[
q_{ij} \coloneqq \sigma_j d_i^{-1} (l_{ji} + r_{ji}), \qquad i \neq j,
\]
and
\[
q_{ii} \coloneqq -\sigma_i d_i^{-1} (l_i + r_i).
\]
Then
\begin{equation}
\label{eq:Q-formula}
Q\phi = \lrp[\Big]{ \sum_{j \neq i} q_{ij} \phi_j + q_{ii} \phi_i }_{i \in \cn},
\qquad \phi = (\phi_i)_{i \in \cn} \in L_0^1(S)
\end{equation}
and the operator \( Q \) may be identified with the matrix
\( (q_{ij})_{i,j \in \cn} \).
(Notice the difference between the matrix defined here and the matrix from
Theorem~\ref{thm:convergence-in-C}.)
The operator \( Q \), since the matrix \( (q_{ij})_{i,j \in \cn} \) is finite,
generates strongly continuous semigroup \( \sem{\e^{tQ}}_{t \geq 0} \) in
\( L_0^1(S) \).

\begin{theorem}
\label{thm:convergence-of-dual}
For each \( n \in \N \) let the operator \( A_n^{*} \) be defined
by~\eqref{eq:A-dual} with domain composed of functions \( \phi \in W^{2,1}(S) \)
satisfying boundary conditions~\eqref{eq:transmission-cond-F}.
Then, for \( P \) and \( Q \) defined by~\eqref{eq:P-formula}
and~\eqref{eq:Q-formula}, respectively, we have
\begin{equation}
\label{eq:limit-semigroup-in-l1}
\lim_{n \to +\infty} \e^{tA_n^{*}} \phi = \e^{tQ} P\phi, \qquad \phi \in L^1(S), \ t >
0
\end{equation}
in \( L^1(S) \).
The convergence is uniform on compact subsets of \( (0,\infty) \).
If \( \phi \in L_0^1(S) \), then~\eqref{eq:limit-semigroup-in-l1} holds also for
\( t = 0 \), and the convergence is uniform on compact subsets of
\( [0,+\infty) \).
\end{theorem}

\begin{proof}
Let \( \X \coloneqq L^1(S) \), \( \ca_n \coloneqq A_n^{*} \),
\( \ca \coloneqq \sigma K \), \( \cB \coloneqq B \), and
\( \epsilon_n \coloneqq \kappa_n^{-1} \).
Then \( \cP \) defined by~\eqref{eq:Q-limit} equals \( P \).
By Lemma~\ref{lem:condition-1-Kurtz} and Lemma~\ref{lem:condition-2-Kurtz}
condistions~\ref{item:Kurtz1} and~\ref{item:Kurtz2} from Kurtz's theorem are
satisfied.
Moreover, \( \cP \ca \) with domain \( \Y \) equals \( Q \).
Therefore, the claim follows by Theorem~\ref{thm:Kurtz}.
\end{proof}

\section{Analysis in \texorpdfstring{\( L^2(S) \)}{L2(S)}}
\label{sec:convergence--l2}

Here we consider a~similar problem as in Section~\ref{sec:analysis--l1}, however
we change the space \( L^1(S) \) to \( L^2(S) \).
Naturally,
\[
L^2(S) \coloneqq \set{ u\colon u = (u_i)_{i \in \cn},\ u_i \in L^2(E_i) },
\]
where \( L^2(E_i) \) is the complex Hilbert space of (equivalence classes of)
Lebesgue square integrable complex functions on \( E_i \), and the latter space
is isometrically isomorphic to the standard \( L^2(0,d_i) \) (see remarks at the
beginning of Section~\ref{sec:adjoint-semigroup}).
In contradistinction to \( L^1(S) \), we denote elements of \( L^2(S) \) by
\( u \) and \( v \).
The space \( L^2(S) \) equipped with the scalar product
\[
\scl{u,v}_{L^2(S)} \coloneqq \int_S u \conj{v} = \sum_{i \in \cn} \int_{E_i} u_i
\conj{v_i} = \sum_{i \in \cn} \scl{u_i,v_i}_{L^2(E_i)}
\]
is a~complex Hilbert space.
By \( H^1(S) \) we denote the Sobolev space \( W^{1,2}(S) \subset L^2(S) \),
that is \( u \in H^1(S) \) if and only if \( u \in L^2(S) \), \( u \) is weakly
differentiable and \( u' \in L^2(S) \).
Similarly we define \( H^2(S) = W^{2,2}(S) \) as the space of \( u \in L^2(S) \)
such that \( u \) and \( u' \) are weakly differentiable, and
\( u', u'' \in L^2(S) \).

For each \( n \in \N \) we define the operator \( A_n^{*} \) in \( L^2(S) \)
similarly as in Section~\ref{sec:convergence}, that is
\[
A_n^{*} u \coloneqq \kappa_n \sigma u'', \qquad u \in D(A_n^{*}),
\]
where \( D(A_n^{*}) \) is the set of function \( u \in H^2(S) \) such that
transmission conditions~\eqref{eq:transmission-cond-F} hold.
Here we consider \( F_{L,j} \) and \( F_{R,j} \) as functionals on \( H^2(S) \).
We prove in Theorem~\ref{thm:An-gen} that \( A_n^{*} \)'s generate holomorphic
semigroups in \( L^2(S) \) and, in Theorem~\ref{thm:agn-asymptotics},
investigate their asymptotics.

\subsection{Sesquilinear forms}
\label{sec:sesquilinear-forms}

In what follows we extensively use the theory of sesquilinear forms, see for
example~\cite[Chapter~6]{kato95} or~\cite[Chapter~1]{ouhabaz05}.
We recall that a~\emph{sesquilinear form} (or simply \emph{form}) in a~complex
Hilbert space \( (H,\scl{\cdot,\cdot}_H) \) is a~mapping
\( \fa\colon D(\fa) \times D(\fa) \to \C \) such that \( \fa(\cdot,u) \) is
linear and \( \fa(u,\cdot) \) is antilinear for all \( u \in D(\fa) \).
The set \( D(\fa) \) is a~linear subspace of \( H \) and is called the
\emph{domain} of \( \fa \).
We say that \( \fa \) is \emph{densely defined} if \( D(\fa) \) is a~dense set
in \( H \), \emph{accretive} if \( \re \fa(u,u) \geq 0 \) for each
\( u \in D(\fa) \), and \emph{closed} if \( D(\fa) \) is a~Hilbert space with
respect to the inner product
\( \scl{u,v}_{\fa} \coloneqq \re \fa(u,v) + \scl{u,v}_H \),
\( u,v \in D(\fa) \).
Moreover, we call \( \fa \) \emph{sectorial} if there exists \( M > 0 \) such
that
\begin{equation}
\label{eq:sectorial}
\abs{ \im \fa(u,u) } \leq M \re \fa(u,u), \qquad u \in D(\fa).
\end{equation}
If \( \seq{\fa_n}_{n \in \N} \) is a~sequence of forms in \( H \), then we say
that forms \( \fa_n \)'s are \emph{uniformly sectorial} if there exists
\( M > 0 \) (independent of \( n \)) such that~\eqref{eq:sectorial} holds with
\( \fa \) replaced by \( \fa_n \) for \( n \in \N \).
Also, to shorten notation, we write \( \fa(u) \) for \( \fa(u,u) \).

For a~densely defined form \( \fa \) we define the \emph{associated operator}
\( A \) in the following way.
The domain \( D(A) \) of \( A \) is the set of \( u \in D(\fa) \) such that
there exists \( f \in D(\fa) \) satisfying
\[
\fa(u,v) = -\scl{f,v}, \qquad v \in D(\fa).
\]
For \( u \in D(\fa) \) we set
\[
Au \coloneqq f.
\]
This definition is correct since by the density of \( D(\fa) \) the element
\( f \) is unique.
It turns out, see~\cite[Theorem~VI.2.1]{kato95}
or~\cite[Theorem~1.52]{ouhabaz05}, that the operator associated with a~densely
defined, accretive, closed and sectorial form \( \fa \) is the generator of
a~bounded holomorphic semigroup in \( H \) denoted
\( \sem{\e^{-t\, \fa}}_{t \geq 0} \).

In order to state Ouhabaz's result (see~\cite[Theorem~5]{ouhabaz95}), which is
our main tool in this section, we need to introduce the notion of the
\emph{degenerate semigroup} related to a~non densely defined form.
Let \( \fa \) be a~form in \( H \).
If the domain \( D(\fa) \) is not dense in \( H \), then there is no operator
associated with the form \( \fa \).
However, we may consider the form in the closure \( H_0 \) of \( D(\fa) \) in
\( H \).
Then \( H_0 \) is a~Hilbert space and there is the operator \( A_0 \) associated
with \( \fa \) as restricted to \( H_0 \).
If the form \( \fa \) is accretive, closed and sectorial, then \( A_0 \)
generates a~bounded, holomorphic semigroup \( \sem{\e^{t A_0}}_{t \geq 0} \) in
\( H_0 \).
We extend this semigroup to the \emph{degenerate semigroup}
\( \sem{\e^{-t\,\fa}}_{t \geq 0} \) in \( H \), by setting
\[
\e^{-t\, \fa} u \coloneqq \e^{t A_0} P_{H_0} u, \qquad u \in H,\ t \geq 0,
\]
where \( P_{H_0} \) is the orthogonal projection of \( H \) onto \( H_0 \).

In our particular setup, we use the following special case of Ouhabaz's theorem
(see~\cite[Theorem~3.2 and Corollary~3.3]{bobkazkun17} for the general version).

\begin{theorem}
\label{thm:ouhabaz}
Let \( \seq{\fa_n}_{n \in \N} \) be a~sequence of accretive, closed and
uniformly sectorial forms defined on the same domain \( D \) in a~Hilbert space
\( H \).
Assume that
\begin{enumerate}
  \item \( \re \fa_n(u) \leq \re \fa_{n+1}(u) \) for every \( u \in D \),
  \item for each \( u \in D \) the imaginary part \/\( \im \fa_n(u) \) does not
depend on \( n \in \N \).
\end{enumerate}
Then the form \( \fa \) defined by
\[
\fa(u,v) \coloneqq \lim_{n \to +\infty} \fa_n(u,v), \qquad u,v \in D(\fa)
\]
with domain
\[
D(\fa) \coloneqq \set[\big]{ u \in D\colon \sup_{n \in \N} \fa_n(u) < +\infty },
\]
is accretive, closed and sectorial.
Moreover, for every \( u \in H \) and \( t > 0 \),
\begin{equation}
\label{eq:ouhabaz-limit}
\lim_{n \to +\infty} \e^{-t\, \fa_n} u = \e^{-t\, \fa} u, \qquad u \in H,\ t > 0
\end{equation}
in \( H \), and the convergence is uniform on compact subsets of
\( (0,\infty) \).
If \( u \) is in the closure of \( D(\fa) \), then~\eqref{eq:ouhabaz-limit}
holds also for \( t = 0 \), and the convergence is uniform on compact subsets of
\( [0,+\infty) \).
\end{theorem}

\subsection{Generation theorem in \texorpdfstring{\( L^2(S) \)}{L2(S)}}
\label{sec:generation-theorem-L2}

We prove a~generation result analogous to Theorem~\ref{thm:dual-is-Markov}.

\begin{theorem}
\label{thm:An-gen}
For each \( n \in \N \) the operator \( A_n^{*} \) in \( L^2(S) \) generates
a~holomorphic semigroup \( \sem{\e^{t A_n^{*}}}_{t \geq 0} \) in \( L^2(S) \).
Furthermore, there exists \( \gamma > 0 \) such that
\begin{equation}
\label{eq:eAn-norm}
\norm{ \e^{tA_n^{*}} }_{\lin(L^2(S))} \leq \e^{\gamma t}, \qquad n \in \N,\ t
\geq 0.
\end{equation}
\end{theorem}

Throughout this section fix \( n \in \N \).
We begin by finding a~form \( \fa_n \) in \( L^2(S) \) such that \( A_n^{*} \)
is the operator associated with \( \fa_n \).
Define the form \( \fb_n \) in \( L^2(S) \) by
\[
\fb_n(u,v) \coloneqq \kappa_n \scl{\sigma u',v'}_{L^2(S)}
\]
with domain \( D(\fb_n) \coloneqq H^1(S) \).
Let \( u \in D(A_n^{*}) \) and \( v \in H^1(S) \).
Integration by parts gives
\begin{equation*}
\begin{split}
\int_{E_i} u'' \conj{v} = -\int_{E_i} u' \conj{v}' + u'(R_i) \conj{v}(R_i) -
u'(L_i) \conj{v}(L_i), \qquad i \in \cn.
\end{split}
\end{equation*}
Hence, since \( u \) satisfies transmission
conditions~\eqref{eq:transmission-cond-F},
\begin{equation}
\label{eq:form-operator-calc}
\scl{A_n^{*}u,v}_{L^2(S)} = - \fb_n(u,v) - \fc(u,v),
\end{equation}
where \( \fc \) is the form in \( L^2(S) \) given by
\[
\fc(u,v) \coloneqq \sum_{i \in \cn} \sigma_i [(F_{L,i} u) \conj{v}(L_i) -
(F_{R,i} u) \conj{v}(R_i)]
\]
with domain \( D(\fc) \coloneqq H^1(S) \).
Note that \( \fc \) does not depend on \( n \).
Formula~\eqref{eq:form-operator-calc} suggests that we should set
\( D(\fa_n) \coloneqq H^1(S) \) and define
\begin{equation}
\label{eq:form-a}
\fa_n \coloneqq \fb_n + \fc.
\end{equation}

The space \( H^1(S) \) is dense in \( L^2(S) \), therefore, in order to prove
that the operator associated with \( \fa_n \) generates a~holomorphic semigroup,
we are left with proving that the form \( \fa_n \) is accretive, closed and
sectorial.

For the proofs of Lemma~\ref{lem:bn} and
Proposition~\ref{prop:a-closed-sectorial} it is useful to denote
\[
\sigma_{\min} \coloneqq \min_{i \in \cn} \sigma_i, \qquad \sigma_{\max}
\coloneqq \max_{i \in \cn} \sigma_i.
\]

\begin{lemma}
\label{lem:bn}
The form \( \fb_n \) is accretive and closed.
\end{lemma}

\begin{proof}
For \( \sqrt{\sigma} \coloneqq (\sqrt{\sigma_i})_{i \in \cn} \) we have
\begin{equation}
\label{eq:bnu}
\fb_n(u) = \kappa_n \norm{ \sqrt{\sigma} u' }_{L^2(S)}^2, \qquad u \in H^1(S),
\end{equation}
which proves accretivity.
Observe that
\[
\kappa_n \sigma_{\min} \norm{ u' }_{L^2(S)}^2 \leq \fb_n(u) \leq \kappa_n
\sigma_{\max} \norm{ u' }_{L^2(S)}, \qquad u \in H^1(S).
\]
Hence the norm \( \norm{ \, \cdot \, }_{\fb_n} \) associated with \( \fb_n \) is
equivalent to the standard norm in \( H^1(S) \) (which is a~Hilbert space), and
the claim follows.
\end{proof}

\begin{proposition}
\label{prop:a-closed-sectorial}
The form \( \fa_n \) is closed and there exists \( \gamma > 0 \) such that the
form \( \fa_n + \gamma \) is sectorial with
\begin{equation}
\label{eq:a-sectorial}
\abs{ \im(\fa_n + \gamma)(u) } \leq \re (\fa_n + \gamma)(u), \qquad u \in H^1(S).
\end{equation}
\end{proposition}

Here, by the form \( \fa_n + \gamma \) we mean the form defined by
\( (\fa_n+\gamma)(u,v) = \fa_n(u,v) + \gamma \scl{u,v}_{L^2(S)} \).

\begin{proof}
Observe that for some \( c > 0 \) we have
\[
\abs{ \fc(u) } \leq c \norm{ u }_{L^{\infty}(S)}^2, \qquad u \in H^1(S),
\]
where \( \dnorm_{L^{\infty}(S)} \) is the standard (essential) supremum norm.
By the Gagliardo-Nirenberg interpolation (see
e.g.~\cite[Theorem~12.83]{giovanni17}) there exists \( C > 0 \) such that
\[
\norm{ u }_{L^{\infty}(S)}^2 \leq C \norm{ u }_{L^2(S)} \norm{ u' }_{L^2(S)},
\qquad u \in H^1(S).
\]
Hence, by Young's inequality,
\[
\abs{ \fc(u) } \leq \frac{\gamma}{2} \norm{ u }_{L^2(S)}^2 + \frac{\kappa_n
  \sigma_{\min}}{2} \norm{ u' }_{L^2(S)}^2, \qquad u \in H^1(S)
\]
for \( \gamma \coloneqq c^2 C^2 / (\kappa_n \sigma_{\min}) \).
Therefore,
\begin{equation}
\label{eq:c-estimate}
\abs{ \fc(u) } \leq \frac{1}{2} \fb_n(u) + \frac{\gamma}{2} \norm{ u
}_{L^2(S)}^2.
\end{equation}
This means that \( \fc \) is \( \fb_n \)-form bounded with \( \fb_n \)-bound
\( 1/2 \) (see~\cite[Definition~1.17]{ouhabaz05}).
Using~\cite[Theorem~VI.3.4]{kato95} or~\cite[Theorem~1.19]{ouhabaz05}, it
follows that the form \( \fa_n = \fb_n + \fc \) is closed as a~relatively
bounded perturbation of the closed form \( \fb_n \).

To show the second part of the lemma notice that by~\eqref{eq:form-a}
and~\eqref{eq:c-estimate} we have
\[
\abs{ \im \fa_n(u) } = \abs{ \im \fc(u) } \leq \frac 12 \fb_n(u) +
\frac{\gamma}{2} \norm{ u }_{L^2(S)}^2,
\]
and
\[
\re \fa_n(u) \geq \fb_n(u) - \abs{ \re \fc(u) } \geq \frac 12 \fb_n(u) -
\frac{\gamma}{2} \norm{ u }_{L^2(S)}^2.
\]
Combining these two inequalities we obtain
\[
\abs{ \im(\fa_n + \gamma)(u) } = \abs{ \im \fa_n(u) } \leq \re \fa_n(u) + \gamma
\norm{ u }_{L^2(S)}^2,
\]
which proves~\eqref{eq:a-sectorial}.
\end{proof}

\begin{proposition}
\label{prop:an-An}
The operator associated with \( \fa_n \) is \( A_n^{*} \).
\end{proposition}

\begin{proof}
Let \( B_n \) be the operator associated with \( \fa_n \).
We claim that \( B_n \) is exactly \( A_n^{*} \).
Formula~\eqref{eq:form-operator-calc} shows that
\( D(A_n^{*}) \subseteq D(B_n) \) and \( B_n u = A_n^{*} u \) for
\( u \in D(A_n^{*}) \).
On the other hand let \( u \in D(B_n) \).
There exists \( f \in H^1(S) \) such that
\( \fa_n(u,v) = - \scl{f,v}_{L^2(S)} \) for all \( v \in H^1(S) \).
Choose \( v \in H^1(S) \) that on each edge is compactly supported smooth
function.
Then \( \fc(u,v) = 0 \) and consequently
\[
\fa_n(u,v) = \fb_n(u,v) = \kappa_n \scl{\sigma u',v'}_{L^2(S)}.
\]
Therefore \( \scl{f,v} = -\kappa_n \scl{\sigma u',v'} \), which proves that
\( u \in H^2(S) \) and \( f = B_n u = \kappa_n \sigma u'' \).
Now for fixed \( i \in \cn \) choose \( v \in H^1(S) \) with
\[
v(R_i) \neq 0, \quad v(L_i) = 0,\qquad \text{and} \qquad v(R_j) = v(L_j) = 0,
\quad j \in \cn,\ j \neq i.
\]
Then
\[
\fc(u,v) = -\sigma_i (F_{R,i} u) \bar v(R_i)
\]
and, integrating by parts,
\[
\fb_n(u,v) = -\kappa_n \scl{\sigma u'',v}_{L^2(S)} + \kappa_n \sigma_i u'(R_i)
\bar v(R_i).
\]
Hence
\begin{equation*}
\fa_n(u,v) = -\kappa_n \scl{\sigma u'',v}_{L^2(S)} + \sigma_i
\lrb{ \kappa_n u'(R_i) - (F_{R,i} u) } \bar v(R_i).
\end{equation*}
This equality, since \( \fa_n(u,v) = -\kappa_n \scl{\sigma u'',v} \), is
equivalent to
\[
\kappa_n u'(R_i) = F_{R,i} u.
\]
In the same way we prove that
\[
\kappa_n u'(L_i) = F_{L,i} u.
\]
This means that transmission conditions~\eqref{eq:transmission-cond-F} are
satisfied and, since \( u \in H^2(S) \), it follows that \( u \in D(A_n^{*}) \).
Finally \( D(B_n) = D(A_n^{*}) \) and \( B_n u = A_n^{*} u \) for all
\( u \in D(A_n^{*}) \).
\end{proof}

\begin{proof}[Proof of Theorem~\textup{\ref{thm:An-gen}}]
Let \( \gamma > 0 \) be as in Proposition~\ref{prop:a-closed-sectorial}.
Then the form \( \fa_n + \gamma \) is densely defined, accretive, closed and
sectorial.
Moreover, by Proposition~\ref{prop:an-An}, the operator associated with
\( \fa_n + \gamma \) is \( A_n^{*} - \gamma \).
Therefore, by~\cite[Theorem~1.52]{ouhabaz05}, it follows that
\( A_n^{*} - \gamma \) generates a~holomorphic contraction semigroup in
\( L^2(S) \).
Hence, \( A_n^{*} \) generates a~holomorphic semigroup in \( L^2(S) \), and
since
\[
\norm{ \e^{tA_n^{*}} \e^{-\gamma} }_{\lin(L^2(S))} \leq 1, \qquad n
\in \N,\ t \geq 0,
\]
inequality~\eqref{eq:eAn-norm} holds.
\end{proof}

\subsection{Convergence result in \texorpdfstring{\( L^2(S) \)}{L2(S)}}
\label{sec:convergence-result-l2}

Let \( L_0^2(S) \) be the closed subspace of \( L^2(S) \) consisting of complex
functions that are constant on each edge.
Similarly as for \( L_0^1(S) \) defined in Section~\ref{sec:convergence}, the
space \( L_0^2(S) \) is isometrically isomorphic to \( \C^N \) equipped with the
appropriate scalar product.

Let \( Q \) be the operator in \( L_0^2(S) \) defined as in \( L_0^1(S) \) by
formula~\eqref{eq:Q-formula}.
Similarly, let \( P \) be the projection of \( L^2(S) \) onto \( L_0^2(S) \)
given by~\eqref{eq:P-formula}.
Then, for the operators \( A_n^{*} \)'s defined in the beginning of
Section~\ref{sec:convergence--l2}, the following analogous result to
Theorem~\ref{thm:convergence-of-dual} holds.

\begin{theorem}
\label{thm:agn-asymptotics}
For every \( u \in L^2(S) \) we have
\begin{equation}
\label{eq:l2-limit}
\lim_{n \to +\infty} \e^{t A_n^{*}} u = \e^{tQ} P u, \qquad t > 0
\end{equation}
in \( L^2(S) \).
The convergence is uniform on compact subsets of \( (0,+\infty) \).
If \( \phi \in L_0^2(S) \), then~\eqref{eq:l2-limit} holds also for \( t = 0 \),
and the convergence if uniform on compact subsets of \( [0,+\infty) \).
\end{theorem}

For \( n \in \N \) let \( \fa_n \) be the form in \( L^2(S) \) defined
by~\eqref{eq:form-a}.
Fix \( \gamma > 0 \) as in Proposition~\ref{prop:a-closed-sectorial} and define
\[
\fa_n^{\gamma} \coloneqq \fa_n + \gamma = \fb_n + \fc + \gamma, \qquad n \in \N
\]
with domain
\[
D(\fa_n^{\gamma}) \coloneqq D(\fa_n) = H^1(S).
\]

\begin{lemma}
\label{lem:angamma-properties}
The sequence \( \seq{\fa_n^{\gamma}}_{n \in \N} \) consists of accretive, closed
and uniformly sectorial forms.
Moreover,
\[
\re \fa_n^{\gamma}(u) \leq \re \fa_{n+1}^{\gamma}(u),
\]
and
\[
\im \fa_n^{\gamma}(u) = \im \fc(u)
\]
for all \( n \in \N \), \( u \in H^1(S) \).
\end{lemma}

\begin{proof}
The first part is a~consequence of Proposition~\ref{prop:a-closed-sectorial}.
For the second observe that
\begin{equation}
\label{eq:agnu}
\fa_n^{\gamma}(u) = \fb_n(u) + \fc(u) + \gamma \norm{ u }_{L^2(S)}^2, \qquad n
\in \N,\ u \in H^1(S).
\end{equation}
The claim follows from the fact that \( \kappa_n \leq \kappa_{n+1} \) and
\( \fb_n(u) = \re \fb_n(u) \) for all \( n \in \N \) and \( u \in H^1(S) \).
\end{proof}

Let \( \fa^{\gamma} \) be the form in \( H \) defined by
\begin{equation}
\label{eq:limit-form}
\fa^{\gamma}(u,v) \coloneqq \lim_{n \to +\infty} \fa_n^{\gamma}(u,v)
\end{equation}
with domain
\[
D(\fa^{\gamma}) \coloneqq \set{ u \in H^1(S)\colon \sup_{n \in \N}
  \fa_n^{\gamma}(u) < +\infty }.
\]
This definition makes sense because the limit of \( \fa_n^{\gamma}(u) \) as
\( n \to +\infty \) exists, and we may define \( \fa^{\gamma} \) by the
polarization equality.

\begin{lemma}
\label{lem:ag-dom}
We have
\[
D(\fa^{\gamma}) = L_0^2(S)
\]
and
\begin{equation}
\label{eq:ag-def}
\fa^{\gamma}(u,v) = \fc(u,v) + \gamma
\scl{u,v}_{L^2(S)}, \qquad u,v \in L_0^2(S).
\end{equation}
\end{lemma}

\begin{proof}
Let \( u \in H^1(S) \) and observe that \( u \in D(\fa^{\gamma}) \) if and only
if
\begin{equation*}
\sup_{n \in \N} \fb_n(u) < +\infty.
\end{equation*}
By~\eqref{eq:bnu}, the last condition holds if and only if \( u' = 0 \) in
\( L^2(S) \), since \( \kappa_n \to +\infty \) as \( n \to +\infty \).
This completes the proof, because \( u' = 0 \) is equivalent to
\( u \in L_0^2(S) \), and the formula~\eqref{eq:ag-def} follows now immediately
from~\eqref{eq:limit-form} and~\eqref{eq:agnu}.
\end{proof}

Let \( \fc_0 \) be the restriction of \( \fc \) to \( L_0^2(S) \), that is the
form in \( L_0^2(S) \) given by
\[
\fc_0(u,v) \coloneqq \fc(u,v), \qquad u,v \in L_0^2(S).
\]

\begin{lemma}
\label{lem:c-associated}
The operator associated with \( \fc_0 \) equals \( Q \).
\end{lemma}

\begin{proof}
Let \( u,v \in L_0^2(S) \).
Then, calculating as in~\eqref{eq:sigmaPK},
\[
\sigma_i (F_{R,i} u - F_{L,i} u) = \sum_{j \neq i} \sigma_j (l_{ji} + r_{ji})
u_j - \sigma_i (l_i + r_i) u_i, \qquad i \in \cn,
\]
where \( u_j \) is the value of \( u \) on the edge \( E_j \).
Therefore, see~\eqref{eq:Q-formula},
\[
\fc_0(u,v) = \sum_{i \in \cn} \sigma_i (F_{R,i} u - F_{L,i} u) \conj{v_i} = -
\scl{Qu,v}_{L^2(S)},
\]
and the claim follows.
\end{proof}

\begin{corollary}
\label{cor:ag-gen}
The operator associated with \( \fa^{\gamma} \), as a~form in \( L_0^2(S) \),
equals \( Q - \gamma \).
\end{corollary}

\begin{proof}
The claim is a~consequence of~\eqref{eq:ag-def} and
Lemma~\ref{lem:c-associated}.
\end{proof}

Finally, we are ready to prove our convergence result in \( L^2(S) \).

\begin{proof}[Proof of Theorem~\textup{\ref{thm:agn-asymptotics}}]
By Lemma~\ref{lem:angamma-properties} the assumptions of
Theorem~\ref{thm:ouhabaz} hold for the sequence
\( \seq{\fa_n^{\gamma}}_{n \in \N} \).
Hence,
\[
\lim_{n \to +\infty} \e^{-t\, \fa_n^{\gamma}} u = \e^{-t\, \fa^{\gamma}} u, \qquad
u \in L_0^2(S).
\]
By Proposition~\ref{prop:an-An}, \( A_n^{*} - \gamma \) is associated with
\( \fa_n^{\gamma} \), and hence by Corollary~\ref{cor:ag-gen} we can rewrite the
above relation in the form
\[
\lim_{n \to +\infty} \e^{t(A_n^{*} - \gamma)} u = \e^{t(Q - \gamma)} P u, \qquad
u \in L_0^2(S),
\]
which is equivalent to~\eqref{eq:l2-limit}.
\end{proof}

\bibliography{diffusions-on-graphs}
\bibliographystyle{amsplain}

\end{document}